\documentclass{amsart}
\usepackage{amssymb}
\usepackage{amsthm}
\theoremstyle{plain}\newtheorem{Theorem}{Theorem}[section]
\theoremstyle{plain}
\theoremstyle{plain}\newtheorem{Corollary}[Theorem]{Corollary}
\theoremstyle{plain}\newtheorem{Lemma}[Theorem]{Lemma}
\theoremstyle{plain}\newtheorem{Proposition}[Theorem]{Proposition}
\theoremstyle{definition}
\theoremstyle{definition}\newtheorem{Example}[Theorem]{Example}
\theoremstyle{definition}
\theoremstyle{definition}\newtheorem{Remark}[Theorem]{Remark}
\theoremstyle{definition}

    \def\OG{{\mathcal{O}G}}  \def\OGb{{\mathcal{O}Gb}}
    \def\OH{{\mathcal{O}H}}  
    \def\OP{{\mathcal{O}P}}

    \def\ON{{\mathcal{O}N}}

\def\CO{{\mathcal{O}}}

\def\Z{{\mathbb Z}}

\def\ann{\mathrm{ann}}
                
             \def\ten{\otimes}

\def\dim{\mathrm{dim}}           
\def\End{\mathrm{End}}

\def\Hom{\mathrm{Hom}}           

\def\ker{\mathrm{ker}}           
\def\IBr{\mathrm{IBr}}         
\def\Id{\mathrm{Id}}             \def\tenA{\otimes_A}
\def\Im{\mathrm{Im}}             \def\tenB{\otimes_B}
\def\Ind{\mathrm{Ind}}

\def\Irr{\mathrm{Irr}}           
           \def\tenO{\otimes_{\mathcal{O}}}

\def\rank{\mathrm{rank}}         
      
\def\Res{\mathrm{Res}}

\def\tr{\mathrm{tr}}

\title{On symmetric quotients of symmetric algebras} 

\author{Radha Kessar} 
\address{City University London, Department of Mathematical Science,
         Northampton Square, London, EC1V 0HB, United Kingdom
        }
\author{Shigeo Koshitani} 
\address{ Department of Mathematics, Graduate School of Science,
          Chiba University, 1-33 Yayoi-cho, Inage-ku, Chiba,
          263-8522, Japan
        }
\author{Markus Linckelmann} 
\address{City University London, Department of Mathematical Science,
         Northampton Square, London, EC1V 0HB, United Kingdom
        }

\begin{document}
\thanks{A part of this work was done while the second author was visiting the other authors in March and October 2013.
The second author gratefully acknowledges support  for these visits  from  the Japan Society for Promotion of Science (JSPS), Grant-in-Aid for Scientific Research (C)23540007, 2011--2014.}

\begin{abstract}
We investigate symmetric quotient algebras of symmetric algebras,
with an emphasis on finite group algebras over a complete discrete 
valuation ring $\CO$. Using elementary methods, we show that if an
ordinary irreducible character $\chi$ of a finite group $G$ gives 
rise to a symmetric quotient over $\CO$ which is not a matrix algebra, 
then the decomposition numbers of the row labelled by $\chi$ are
all divisible by the characteristic $p$ of the residue field of $\CO$.
\end{abstract}

\maketitle

\section{Introduction}

Let $p$ be a prime and $\CO$ a complete discrete valuation ring
having a residue field $k$ of characteristic $p$ and a quotient
field $K$ of characteristic zero. Unless stated otherwise, we
assume that $K$ and $k$ are splitting fields for all finite groups
under consideration. Let $G$ be a finite group.
Any subset $M$ of the set $\Irr_K(G)$ of irreducible $K$-valued 
characters of $G$ gives rise to an $\CO$-free quotient algebra, 
namely the image of a structural homomorphism $\OG\to$ $\End_\CO(V)$, 
where $V$ is an $\CO$-free $\OG$-module having character 
$\sum_{\chi\in M} \chi$. This
image is isomorphic to $\OG(\sum_{\chi\in M} e(\chi))$, where
$e(\chi)$ denotes the primitive idempotent in $Z(KG)$ corresponding
to $\chi$. Any $\CO$-free quotient algebra of $\OG$ arises in this
way; in particular, $\OG$ has only finitely many $\CO$-free
algebra quotients. Any quotient of $\OG$ admits a decomposition
induced by the block decomposition, and hence finding symmetric
quotients of $\OG$ is equivalent to finding symmetric quotients
of the block algebras of $\OG$. We denote by $\IBr_k(G)$ the set
of irreducible Brauer characters of $G$, and by
$d_G : \Z\Irr_K(G)\to$ $\Z\IBr_k(G)$ the
decomposition map, sending a generalised character of $G$
to its restriction to the set $G_{p'}$ of $p'$-elements in $G$. 
For $B$ a block algebra of $\OG$,
we denote by $\Irr_K(B)$ and $\IBr_k(B)$ the sets of irreducible
$K$-characters and Brauer characters, respectively, associated with
$B$. We denote by $d_B : \Z\Irr_K(B)\to$ $\Z\IBr_k(B)$ the
decomposition map obtained from restricting $d_G$.
We denote by $D_G=$ $(d^\chi_\varphi)$ the decomposition matrix
of $\OG$, with rows indexed by $\chi\in$ $\Irr_K(G)$ and columns
indexed by $\varphi\in\IBr_k(G)$; that is, the $d^\chi_\varphi$ are the 
nonnegative integers satisfying $d_G(\chi)=$
$\sum_{\varphi\in\IBr_k(G)}\ d^\chi_\varphi\cdot \varphi$, for
any $\chi\in$ $\Irr_K(G)$. Equivalently, $d^\chi_\varphi=$ $\chi(i)$,
where $i$ is a primitive idempotent in $\OG$ such that $\OG i$ is
a projective cover of a simple $kG$-module with Brauer character
$\varphi$. If $\chi$, $\varphi$ belong to 
different blocks, then $d^\chi_\varphi=$ $0$. For $B$ a block
algebra of $\OG$, we denote by $D_B$ the submatrix of $D_G$
labelled by $\chi\in$ $\Irr_K(B)$ and $\varphi\in$ $\IBr_k(B)$.
We say that $\chi\in$ $\Irr_K(G)$ lifts
the irreducible Brauer character $\varphi\in$ $\IBr_k(G)$ if
$d_G(\chi)=$ $\varphi$, or equivalently, if $d^\chi_\varphi=$ $1$
and $d^\chi_{\varphi'}=0$ for all $\varphi'\in$ $\IBr_k(G)$ 
different from $\varphi$. In that case, $\OG e(\chi)$ is a matrix
algebra over $\CO$, hence trivially symmetric (see 
\ref{OGchimatrix} below).

\begin{Theorem} \label{decnumberp}
Let $G$ be a finite group and $\chi\in$ $\Irr_K(G)$.
Suppose that $\OG e(\chi)$ is symmetric. Then either
$\chi$ lifts an irreducible Brauer character, or $d^\chi_\varphi$
is divisible by $p$ for all $\varphi\in$ $\IBr_k(G)$.
\end{Theorem}

This is a special case of the slightly more general 
result \ref{Adecnumberp} below. We note some immediate
consequences.      

\begin{Corollary} \label{blockdecpprime}
Let $G$ be a finite group and $\chi\in$ $\Irr_K(G)$ such that 
$d^\chi_\varphi$ is prime to $p$ for some $\varphi\in$ $\IBr_k(G)$. 
Then the $\CO$-algebra $\OG e(\chi)$ is symmetric if and only
if $\chi$ lifts $\varphi$.
\end{Corollary}

By a result of Dade \cite{Dadecyclic}, all decomposition numbers of 
blocks with cyclic defect are either $0$ or $1$. 

\begin{Corollary} \label{cyclicdefectCor}
Let $G$ be a finite group and $B$ a block with cyclic defect groups.
Let $\chi\in$ $\Irr_K(B)$. Then $\OG e(\chi)$ is symmetric
if and only if $\chi$ corresponds to a nonexceptional vertex at the 
end of a branch of the Brauer tree of $B$.
\end{Corollary}

By Erdmann's results in \cite{Erdmanntame},  every row of the 
decomposition  matrix of a  nonnilpotent tame block  contains   
an  entry  equal to $1$.

\begin{Corollary} \label{tameblockCor}
Suppose that $p=$ $2$. Let $G$ be a finite group and $B$ a
nonnilpotent block of $\OG$ having a defect group $P$ which 
is either  dihedral,   generalised quaternion, or semidihedral.
Then for any $\chi\in$ $\Irr_K(B)$,
the algebra $\OG e(\chi)$ is symmetric if and only if
$\chi$ lifts an irreducible Brauer character in $\IBr_k(B)$.
\end{Corollary}

Many blocks of quasi-simple groups also seem to have the property that   
every row has an entry prime to $p$. 
Any row of the decomposition matrix of a finite group corresponding 
to a height zero character has at least one entry which is prime to $p$.  
Hence \ref{blockdecpprime} yields the following.

\begin{Corollary} \label{blockheightzero}
Let $G$ be a finite group, $B$ a block of $\OG$,  and $\chi\in$ 
$\Irr_K(B)$. Suppose that $\chi$ has height zero. 
Then the $\CO$-algebra $\OG e(\chi)$ is symmetric if and only
if $\chi$ lifts an irreducible Brauer character.
\end{Corollary}
 
By the main result of the first author and Malle \cite{KessarMalle},  
every irreducible character of a finite group lying in a $p$-block  
with abelian defect groups has height zero (the proof in
\cite{KessarMalle} requires the classification of finite simple groups).  
Hence any irreducible character of a finite group lying in a $p$-block 
with abelian defect groups gives rise to a symmetric quotient if and 
only if the character lifts an irreducible Brauer character.

The symmetric algebras arising in the  corollaries  above  are matrix 
algebras (see \ref{OGchimatrix} below).  By contrast, the symmetric 
algebras obtained from nonlinear characters in the next result are not 
isomorphic to matrix algebras.

\begin{Proposition} \label{cyclicindexp}
Let $P$ be a finite $p$-group having a normal cyclic subgroup
of index $p$. Then $\OP e(\chi)$ is symmetric for any
$\chi\in$  $\Irr_K(P)$.
\end{Proposition}

This will be proved in \S \ref{cyclicindexpSection}. 
Since a nilpotent block is isomorphic to a matrix algebra over
one of its defect group algebras, this proposition, specialised
to $p=$ $2$, has the following consequence (which includes
nilpotent tame blocks). 

\begin{Corollary} \label{nilpotenttameblock}
Suppose that $p=$ $2$. Let $G$ be a finite group and $B$ a
nilpotent block of $\OG$ having a defect group $P$ which 
is either  dihedral, generalised quaternion or semidihedral. 
Then for any $\chi\in$ $\Irr_K(B)$, the algebra $\OG e(\chi)$ 
is symmetric.
\end{Corollary}

Further examples of characters $\chi$ with symmetric quotient $\OG e(\chi)$
which are not isomorphic to matrix algebras can be obtained from
characters of central type. An irreducible character $\chi$ 
of a finite group $G$ is {\it of central type} if $\chi(1)^2=$ $|G:Z(G)|$.

\begin{Proposition} \label{centraltype}
Let $G$ be a finite group and $\chi\in$ $\Irr_K(G)$ a character
of central type. Then $\OG e(\chi)$ is symmetric.
\end{Proposition}

This is shown as a special case of a slightly more general situation
in \ref{GNeta}.  As a consequence of \ref{cyclicindexp} and 
\ref{centraltype}, if $P$ is a finite $p$-group of order at most $p^3$, 
then $\OP e(\chi)$ is symmetric for all $\chi\in$ $\Irr_K(P)$. 
In \S \ref{ExamplesSection} we give an example showing that
this is not the case in general for irreducible characters of finite 
$p$-groups of order $p^4$.

\medskip

As is explained in  the remarks following  \ref{Aechisymm},  for $G$ a 
finite group, the set of ideals $I$ of $\OG$ such that $\OG/I$ is 
$\CO$-free corresponds bijectively to the set of subsets of $\Irr_K(G)$.      
If $\chi$ is an ordinary irreducible character of $G$ then 
$\OG e(\chi )$ is the quotient of $\OG$ by the ideal which corresponds 
under the above bijection to the subset $\Irr_K(G) -\{\chi\}$  of 
$\Irr_K(G)$. The next two propositions consider the complementary case    
of quotients $\OG /I $, where $I$ corresponds to a one element subset of 
$\Irr_K(G)$. We first consider the case that $I$ corresponds to the 
trivial character of $G$, in which case $I$ can be explicitly described 
as  $\CO(\sum_{x\in G}x)$. The hypothesis of $K$, $k$ being large enough 
is not necessary for the following result.

\begin{Proposition} \label{IOGpure}
Let $G$ be a finite group. The following are equivalent.

\smallskip\noindent (i)
The $\CO$-algebra $\OG/\CO(\sum_{x\in G}x)$ is symmetric.

\smallskip\noindent (ii)
The group $G$ is $p$-nilpotent and has cyclic Sylow-$p$-subgroups.
\end{Proposition}

\begin{Corollary} \label{IZGpure}
Let $G$ be a finite group. The $\Z$-algebra $\Z G/\Z(\sum_{x\in G}x)$
is symmetric if and only if $G$ is cyclic.
\end{Corollary}

The arguments used in the proof of \ref{IOGpure} can be adapted
to yield the following block theoretic version; we need $k$ 
to be large enough for the block $B$ in the next result. 

\begin{Proposition} \label{Bechipure}
Let $G$ be a finite group, $B$ a block algebra of $\OG$, and let
$\chi\in$ $\Irr_K(B)$. Suppose that $\chi$ lifts an irreducible
Brauer character. 
Set $I=$ $B\cap (K\tenO B)e(\chi)$, where we identify $B$ with its
image $1\ten B$ in $K\tenO B$. The following are equivalent.

\smallskip\noindent (i)
The algebra $B/I$ is symmetric.

\smallskip\noindent (ii)
The block $B$ is nilpotent with cyclic defect groups.
\end{Proposition}

\section{Notation and basic facts} \label{backgroundSection}

If $A$ is an $\CO$-algebra which is free of finite rank as an
$\CO$-module, we denote by $\Irr_K(A)$ the set of characters
of the simple $K\tenO A$-modules. 
Taking characters of $K\tenO A$-modules yields an isomorphism between 
the Grohendieck group $R_K(A)$ of finitely generated $K\tenO A$-modules 
and the free abelian group with basis $\Irr_K(A)$, inducing a 
bijection between the isomorphism classes of simple $K\tenO A$-modules
and $\Irr_K(A)$. If in addition the $K$-algebra $K\tenO A$ is
semisimple, hence a direct product of simple $K$-algebras
corresponding to the isomorphism classes of simple $K\tenO A$-modules,
we denote by $e(\chi)$ the primitive idempotent of $Z(K\tenO A)$
which acts as identity on the simple $K\tenO A$-modules with
character $\chi$ and which annihilates all other simple 
$K\tenO A$-modules. In this case we have $K\tenO A\cong$
$\prod_{\chi\in\Irr_K(A)}\ (K\tenO A)e(\chi)$, and each factor
$(K\tenO A)e(\chi)$ is a simple $K$-algebra. 
An $\CO$-algebra $A$ is {\it symmetric} if $A$ is finitely generated
projective as an $\CO$-module and if $A$ is isomorphic to its
$\CO$-dual $A^\vee =$ $\Hom_\CO(A,\CO)$ as an $A$-$A$-bimodule (this
definition makes sense with $\CO$ replaced by an arbitrary commutative ring).
The image $s$ in $A^\vee$ of $1_A$ under some bimodule isomorphism
$A\cong$ $A^\vee$ is called a {\it symmetrising form of} $A$.
If $s$, $s'$ are two symmetrising forms of $A$, then $s'=$ $z\cdot s$
for a unique $z\in$ $Z(A)^\times$, because any bimodule automorphism
of $A$ is given by left or right multiplication with an invertible
element in $Z(A)$. The ring $\CO$ has the property that every
finitely generated projective $\CO$-module is free. In particular, 
if $A$ is a symmetric $\CO$-algebra, then $A$ is free as an $\CO$-module,
and hence any symmetrising form $s\in$ $A^\vee$ of $A$ satisfies $s(A)=$ 
$\CO$. Let $A$ be an $\CO$-algebra which is finitely generated as an 
$\CO$-module. Since $\CO$ is Noetherian, a submodule $V$ of a finitely 
generated $A$-module $U$ is $\CO$-{\it pure} if and only if $V$ is a 
direct summand of $U$ as an $\CO$-module (see e.g \cite[\S 4J]{Lam} for 
more details on the notion of pure submodules over arbitrary commutative 
rings). We use without further comments the following well-known
facts. See e. g. \cite[17.2]{Feit}, and also \cite[Theorem 1]{Thompson67}
for an application in the context of blocks with cyclic defect groups.

\begin{Lemma} \label{OpureLemma}
Let $A$ be an $\CO$-algebra and let $U$ be an $A$-module which is finitely
generated free as an $\CO$-module. Let $V$ be a submodule of $U$.
Then $U\cap (K\tenO V)$ is the unique minimal $\CO$-pure submodule
of $U$ containing $V$, where we identify $U$ with its image $1\ten U$
in $K\tenO U$. Moreover, the following are equivalent.

\smallskip\noindent (i)
The $A$-module $V$ is $\CO$-pure in $U$.

\smallskip\noindent (ii)
The $A$-module $U/V$ is $\CO$-free.

\smallskip\noindent (iii)
We have $J(\CO)V= J(\CO)U\cap V$.

\smallskip\noindent (iv)
The image of $V$ in $k\tenO U$ is isomorphic to $k\tenO V$.

\smallskip\noindent (v)
We have $V= U\cap (K\tenO V)$, where we identify $U$ to its image $1\ten U$ in 
$K\tenO U$.
\end{Lemma}

Thus if $I$ is an ideal in an $\CO$-algebra $A$ which is finitely
generated free as an $\CO$-module, then the quotient algebra $A/I$ is
$\CO$-free if and only if $I$ is $\CO$-pure in $A$. Any $\CO$-pure ideal 
$I$ of $A$ is equal to $A\cap M_I$ for a unique ideal $M_I$ in $K\tenO A$, 
where we have identified $A$ to its canonical image $1_K\ten A$ in 
$K\tenO A$. If $K\tenO A$ is in addition semisimple, hence a direct product of
simple algebras, then every ideal of $K\tenO A$ is a product of a
subset of those simple algebras. In particular, in that case the set of 
$\CO$-pure ideals in $A$ is finite and corresponds bijectively to the set 
of subsets of a set of representatives of the isomorphism classes
of simple $K\tenO A$-modules.

\medskip
The next result holds with $\CO$ replaced by an arbitrary
commutative Noetherian ring. For symmetric algebras over a field, this 
result is due to Nakayama \cite{Nak39}. The generalisation to algebras 
over commutative Noetherian rings is straightforward (we include a proof 
for the convenience of the reader). Note that the left and right 
annihilators of an ideal $I$ in a symmetric algebra $A$ are always 
equal, denoted by $\ann(I)$. 

\begin{Proposition} [{cf. \cite[Theorem 13]{Nak39}}]
\label{symquotients}
Let $A$ be a symmetric $\CO$-algebra, and let $I$ be an $\CO$-pure ideal
in $A$. The quotient algebra $A/I$ is a symmetric $\CO$-algebra if
and only if there is an element $z\in$ $Z(A)$ such that $\ann(I)=$ $Az$.
\end{Proposition}

\begin{proof}
Set $\bar A=$ $A/I$, and denote by $\pi : A\to$ $\bar A$ the canonical
surjection. Since $I$ is $\CO$-pure, it follows that $A/I$ is finitely
generated projective as an $\CO$-module. Let $s : A\to$ $\CO$ be a 
symmetrising form for $A$; that is, $s$ is symmetric and the map sending 
$a\in$ $A$ to the form $a\cdot s$ defined by $(a\cdot s)(b)=$ $s(ab)$ 
for all $b\in$ $A$ is an isomorphism $A\cong$ $A^\vee=$ $\Hom_\CO(A,\CO)$.
(Since $s$ is symmetric, this map is automatically a homomorphism
of $A$-$A$-bimodules.)  
Suppose that $\ann(I)=$ $Az$ for some $z\in$ $Z(A)$.
Then $t=$ $z\cdot s$ annihilates $I$, hence induces a form
$\bar t\in$ $\bar A$ defined by $\bar t(\bar a)=$ $s(za)$,
where $a\in$ $A$ and $\bar a=$ $a+I\in$ $\bar A$. Since $s$ is
symmetric and $z\in$ $Z(A)$, the forms $t$ and $\bar t$ are again
symmetric. It suffices to show that the map sending $\bar a\in$ $\bar A$
to $\bar a\cdot \bar t\in$ $\bar A^\vee$ is surjective. Let
$\bar u\in$ $\bar A^\vee$. Then $u=$ $\bar u\circ \pi\in$ $A^\vee$, hence
$u=$ $a\cdot s$ for a uniquely determined element $a\in$ $A$.
Since $I\subseteq$ $\ker(u)$, we have $s(aI)=$ $\{0\}$, hence
$a\in$ $\ann(I)$. Thus $a=$ $cz$ for some $c\in$ $A$. It follows that
$u=$ $c\cdot(z\cdot s)=$ $c\cdot t$, and hence $\bar u=$ 
$\bar c\cdot \bar t$, which shows that $\bar A$ is symmetric. 
Suppose conversely that $\bar A$ is
symmetric. Let $\bar t : \bar A\to$ $\CO$ be a symmetrising form, and
set $t=$ $\bar t\circ \pi$. Then $t$ is symmetric (because $\bar t$ is)
and hence $t=$ $z\cdot s$ for some $z\in$ $Z(A)$. Since $I\subseteq$
$\ker(t)$ we have $s(zI)=$ $\{0\}$, hence $z\in$ $\ann(I)$. Let
$b\in$ $\ann(I)$. Then $u=$ $b\cdot s$ has $I$ in its kernel, hence
induces a form $\bar u$ on $\bar A$ satisfying $\bar u(\bar a)=$
$s(ab)$. Since $\bar A$ is symmetric, there is $\bar c\in$ $\bar A$
such that $\bar u=$ $\bar c\cdot \bar t$, hence such that
$u=$ $c\cdot t=$ $(cz)\cdot s$. Thus $(cz)\cdot s=$ $b\cdot s$, and
hence $b=$ $cz\in$ $Az$, whence the equality $\ann(I)=$ $Az$.
\end{proof}

Let $A$ be a symmetric $\CO$-algebra. If $z\in$ $Z(A)$ such that
$Az$ is the annihilator of an $\CO$-pure ideal $I$ in $A$, then $Az$ is
$\CO$-pure. Thus
finding symmetric quotients of $A$ is equivalent to finding elements
$z\in$ $Z(A)$ with the property that $Az$ is $\CO$-pure in $A$.

\begin{Corollary} \label{AzpureCor}
Let $A$ be a symmetric $\CO$-algebra and $z\in$ $Z(A)$.
If $Az$ is $\CO$-pure, then the annihilator $I=$ $\ann(z)=$ $\ann(Az)$
is an $\CO$-pure ideal satisfying $\ann(I)=$ $Az$, and the
$\CO$-algebra $A/I$ is symmetric. Moreover, any symmetric
$\CO$-algebra quotient of $A$ arises in this way.
\end{Corollary}

\begin{proof}
This follows from \ref{symquotients} and the preceding remarks.
\end{proof}

\begin{Proposition} \label{GNeta}
Let $G$ be a finite group and $N$ a normal subgroup of $G$.
Suppose that $K$ is a splitting field for $N$. Let
$\eta\in$ $\Irr_K(N)$, and suppose that $\eta$ is $G$-stable.
If $\ON e(\eta)$ is symmetric, then $\OG e(\eta)$ is symmetric.
\end{Proposition}

\begin{proof}
Suppose that $\ON e(\eta)$ is symmetric. By 
\ref{symquotients} there exists 
an element $z\in$ $Z(\ON)$ such that $\ON z$ is the
annihilator of the kernel of the map $\ON \to$ $\ON e(\eta)$.
Thus $z\in$ $Z(\ON)e(\eta)=$ $\CO e(\eta)$, where we use that
$K$ is large enough for $N$. In particular,
there is $\lambda\in$ $\CO$ such that $z=$ $\lambda e(\eta)$,
and hence $z\in$ $Z(\OG)$.
Since $\OG$ is free as a right $\ON$-module and $\ON z$ is
$\CO$-pure in $\ON$, it follows that
$\OG z$ is $\CO$-pure in $\OG$. The annihilator of the
kernel of the map $\OG\to$ $\OG e(\eta)$ is therefore equal 
to $\OG z$. The result follows from \ref{symquotients}.
\end{proof}

\begin{proof}[Proof of Proposition \ref{centraltype}]
Since $\chi$ is of central type, there is a unique
(hence $G$-stable) linear character $\zeta : Z(G)\to$ $\CO^\times$
such that $e(\chi)=$ $e(\zeta)$. We have $\CO Z(G)e(\zeta)\cong$ 
$\CO$, which is trivially symmetric, and hence $\OG e(\zeta)=$
$\OG e(\chi)$ is symmetric by \ref{GNeta}.
\end{proof}

The argument in the proof of \ref{GNeta} to describe a central element 
$z$ which generates a pure ideal admits the following generalisation.

\begin{Proposition} \label{Aechisymm}
Let $A$ be a symmetric $\CO$-algebra such that $K\tenO A$ is 
semisimple. Let $\chi\in$ $\Irr_K(A)$ be the character of an
absolutely simple $K\tenO A$-module, and denote by $e(\chi)$ the
corresponding primitive idempotent in $Z(K\tenO A)$. Let $\lambda\in$  
$\CO$ be an element having minimal valuation such that 
$\lambda e(\chi)\in$ $A$, where we identify $A$ with its image $1\ten A$
in $K\tenO A$. The following are equivalent.

\smallskip\noindent (i)
The $\CO$-algebra $Ae(\chi)$ is symmetric.

\smallskip\noindent (ii)
The $A$-module $A\lambda e(\chi)$ is $\CO$-pure in $A$.

\smallskip\noindent (iii)
We have $A\lambda e(\chi)=$ $A\cap (K\tenO A)e(\chi)$.
\end{Proposition}

\begin{proof}
The equivalence of (ii) and (iii) is a general fact 
(cf. \ref{OpureLemma}).
Let $I$ be the kernel of the algebra homomorphism $A\to$ $Ae(\chi)$
sending $a\in$ $A$ to $a e(\chi)$. Multiplication by $e(\chi)$ in
$K\tenO A$ yields the projection of $K\tenO A=$ 
$\prod_{\psi\in\Irr_K(A)}\ (K\tenO A)e(\psi)$ onto the factor 
$(K\tenO A)e(\chi)$; this is a matrix algebra as $\chi$ is the
character of an absolutely simple $K\tenO A$-module. Thus $I$ is the 
$\CO$-pure ideal corresponding to the complement of $\{\chi\}$ in $\Irr_K(A)$; 
that is, $I=$ $A\cap \prod_{\psi}\ (K\tenO A) e(\psi)$, where $\psi$
runs over the set $\Irr_K(A)-\{\chi\}$.
It follows that the annihilator of $I$ is equal
to $A\cap (K\tenO A)e(\chi)$, and hence $Z(A)\cap (K\tenO A)e(\chi)\subseteq$
$\CO e(\chi)$. Thus $A/I$ is symmetric if and only if $Az$ is
$\CO$-pure for some $z\in$ $A\cap \CO e(\chi)$, hence if and only if
$A\lambda e(\chi)$ is $\CO$-pure for some $\lambda\in$ $\CO$.
In that case, $\lambda$ has necessarily the smallest possible 
valuation such that $\lambda e(\chi)\in$ $A$. The result follows.
\end{proof}

\begin{Remark} \label{AEndRemark}
Let $A$ be an $\CO$-algebra which is finitely generated free as an
$\CO$-module, and let $I$ be an $\CO$-pure ideal in $A$. Then the image 
of the canonical map $A\to$ $\End_\CO(A/I)$ sending $a\in$ $A$ to left 
multiplication by $a+I$ in $A/I$ has kernel $I$, hence image isomorphic 
to $A/I$. Thus any $\CO$-free algebra quotient of $A$ is isomorphic to 
the image of the structural homomorphism
$A \to \End_\CO(V)$
sending $a\in$ $A$ to left multiplication by $a$ on $V$, for some 
$A$-module $V$ which is free of finite rank as an $\CO$-module.
Since $V$ is $\CO$-free, this image is isomorphic to the canonical
image of $A$ in $\End_K(K\tenO V)$. Thus, if $K\tenO V$ is
a semisimple $K\tenO A$-module, then this image depends only on
the isomorphism classes of simple $K\tenO A$-modules occurring in
a decomposition of $K\tenO V$, but not on the multiplicity of
the simple factors of $K\tenO V$.
\end{Remark}

\begin{Remark} \label{MoritaRemark}
It follows from the formal properties of Morita equivalences that
a Morita equivalence between two algebras induces a bijection
between quotients of the two algebras, in such a way that
quotients corresponding to each other are again Morita equivalent.
In particular, symmetric quotients are preserved under Morita 
equivalences. We describe this briefly for the convenience of
the reader. Let  $A$, $B$ be Morita equivalent 
$\CO$-algebras; that is, there is an $A$-$B$-bimodule $M$ and a 
$B$-$A$-bimodule $N$ such that $M$, $N$ are finitely generated 
projective as left and right modules, and such that we have isomorphisms
of bimodules $M\tenB N\cong$ $A$ and $N\tenA M\cong$ $B$. The functor
$M\tenB -$ induces an equivalence between the categories of
$B$-$B$-bimodules and of $A$-$B$-bimodules, sending $B$ to $M$.
Thus this functor induces a bijection between ideals in $B$ and
subbimodules of $M$, sending an ideal $J$ in $B$ to the subbimodule
$MJ$. Note that since $M$ is finitely generated projective as a 
right $B$-module, we have $MJ\cong$ $M\tenB J$. Similarly, we have
a bijection between ideals in $A$ and subbimodules in $M$ sending
an ideal in $I$ in $A$ to the subbimodule $IM\cong$ $I\tenA M$. 
Combining these bijections yields a bijection between ideals in $A$ and 
ideals in $B$, with the property that the ideal $I$ in $A$ corresponds 
to the ideal $J$ in $B$ if and only if $IM=$ $MJ$, which in turn holds 
if and only if $JN=$ $NI$. If $I$, $J$ correspond to each other through
this bijection, then $M/IM$ and $N/JN$ induce a Morita equivalence
between $A/I $ and $B/J$. Indeed, $M/IM$ is finitely generated
projective as a left $A/I$-module, since $M$ is finitely generated as a 
left $A$-module. Using $IM=$ $MJ$, this argument shows that $M/IM$ is 
finitely generated projective as a right $B/J$-module. Similarly, $N/JN$ 
is finitely generated projective as a left and right module. Moreover, 
we have $M/IM\ten_{B/J} N/JN\cong$ $M/MJ\tenB N/JN\cong$ 
$M\tenB N/JN\cong$ $M\tenB N/M\tenB JN\cong$ $M\tenB N/M\tenB NI\cong$ 
$A/I$. The same argument with reversed roles yields 
$N/NI\ten_{A/I} M/IM\cong$ $B/J$. Since Morita equivalences preserve the 
property of being symmetric, this shows that $A/I$ is symmetric if and 
only if $B/J$ is symmetric.
\end{Remark}

\section{Proof of Proposition \ref{cyclicindexp}} \label{cyclicindexpSection}

Let $G$ be a finite $p$-group having a cyclic normal subgroup $H$ of 
index $p$, and let $\chi\in$ $\Irr_K(G)$. 
In order to prove Proposition \ref{cyclicindexp} we may assume
that $G$ is nonabelian, hence $H$ has order at least $p^2$. 

\medskip
Suppose first that $p$ is odd. Then the automorphism of $H$ induced by
an element $t\in$ $G-H$ acts trivially on the subgroup $H^p$ of
index $p$ in $H$, hence $Z(G)=$ $H^p$ has index $p^2$ in $G$.
Any nonlinear character of $G$ has degree $p$, hence is a
character of central type. It follows from \ref{centraltype} 
that  $\OG e(\chi)$ is symmetric.

\medskip
Suppose now that $p=$ $2$. The previous arguments 
remain valid so long as the action of $G$ on the cyclic normal
$2$-subgroup $H$ of index $2$ is trivial on the subgroup $H^2$
of index $2$ in $H$. This includes the case of 
semidihedral $2$-groups (where $t$ is an  involution which acts on
the cyclic subgroup $H$ of order $2^n$ by sending a generator $s$ 
of $H$ to $s^{1+2^{n-1}}$). If $n=2$, then
$|G|=$ $8$, hence $\chi$ is a character of central type, and so the 
symmetry of $\OG e(\chi)$ follows from \ref{centraltype}. 
Suppose that $n\geq 3$ and that $G$ does not act trivially on the subgroup 
of index $2$ in $H$.  Then $G$ is  dihedral or generalised
quaternion (corresponding in both cases to the action of an element $t \in$ 
$G-H$  of order either $2$ or $4$ on $H$  sending $s$ to $s^{-1}$) or 
quasidihedral  (corresponding to the action of $t$ sending $s$ to 
$s^{-1+2^{n-1}}$).
We will need the following elementary facts (a proof is included
for the convenience of the reader). 

\begin{Lemma} \label{zetaazetab}
Let $n\geq$ $3$ and let $\zeta$ be a primitive $2^n$-th root of unity in $\CO$.
Let $a$, $b\in$ $\Z$ such that $b-a$ is even. The following hold.

\smallskip\noindent  (i) 
The numbers $\zeta^a\pm \zeta^b$ are divisible by $(1-\zeta)^2$  in $\CO$.

\smallskip\noindent  (ii) 
The numbers $\frac{\zeta\pm \zeta^{-1}}{(1-\zeta)^2}$ are invertible in $\CO$.
\end{Lemma}

\begin{proof}
The integer $2$ is divisible by $(1-\zeta)^3$ in $\CO$, as $n\geq 3$. In 
particular, $\frac{2}{(1-\zeta)^2}\in$ $J(\CO)$, and hence, in order to prove
(i),  it suffices to  show that $\zeta^a+\zeta^b$ is divisible by $(1-\zeta)^2$. 
Write $b-a=$ $2c$ for some integer $c$. Then $\zeta^a+\zeta^b=$ 
$\zeta^a(1+\zeta^{b-a})=$ $\zeta^a(1+\zeta^{2c})$. Thus we may 
assume $a=$ $0$ and $b=$ $2c$. We have $1+\zeta^{2c}=$ 
$1-\zeta^{2c}+2\zeta^{2c}=$ $(1-\zeta^c)(1+\zeta^c)+2\zeta^{2c}=$ 
$(1-\zeta^c)(1-\zeta^c+2\zeta^c)+2\zeta^{2c}$, wich is divisible by 
$(1-\zeta)^2$  because $1-\zeta$ divides $1-\zeta^c$ and 
$(1-\zeta)^2$ divides $2$. This shows (i).   Again, since  $\frac{2}{(1-\zeta)^2}\in$ $J(\CO)$, in order to prove (ii), 
it suffices to prove that   $\frac{\zeta - \zeta^{-1}}{(1-\zeta)^2}$  is invertible in $\CO$.   Observe that 
$\frac{1-\zeta^2}{(1-\zeta)^2}=$ $\frac{1+\zeta}{1-\zeta}=$
$1+\frac{2\zeta}{1-\zeta}$ is invertible in $\CO$. Multiplying this by
$\zeta^{-1}$ shows (ii).
\end{proof}

\medskip
We complete the proof of \ref{cyclicindexp}. 
Let $G$ be a  dihedral or generalised  quaternion group of order $2^{n+1}$. 
In order to show that $\OG e(\chi)$ is 
symmetric, we may assume that  $\chi(1)>1$, hence $\chi(1)=$ $2$. Let 
$H$ be the cyclic subgroup of order $2^n$ of $G$.
Since $\chi(1)=$ $2$ we have $\chi=$ $\Ind^G_H(\eta)$ for some
nontrivial linear character $\eta\in$ $\Irr_K(H)$. Arguing by induction, 
we may assume that $\eta$ is faithful.  Let $\zeta$ be a root of
unity of order $2^n$ and let $s$ be a generator of $H$ such that 
$\eta(s)=$ $\zeta$. Let $\bar\eta$ be the character of $H$ sending
$s$ to $\zeta^{-1}$. Then $\Res^G_H(\chi)=$ $\eta+\bar\eta$, and
$\chi$ vanishes outside $H$. In particular, 
$$e(\chi)=e(\eta)+e(\bar\eta) = \frac{1}{2^n}\sum_{a=0}^{2^n-1}
(\zeta^a+\zeta^{-a})s^a\ .$$ 
By \ref{zetaazetab} the  coefficients $\zeta^a+\zeta^{-a}$ in this sum 
are all divisible by $(1-\zeta)^2$.
Set $\lambda=$ $\frac{2^n}{(1-\zeta)^2}$, and set $z=$
$\lambda e(\chi)$. By    \ref{zetaazetab}(i)  we have $z\in$ $\OG$ and by 
\ref{zetaazetab}(ii)   we have that $\lambda $ has minimal valuation such that $\lambda e(\chi)\in \OG $.
Thus, by  \ref{Aechisymm}, it suffices to show that $\OG z$ is $\CO$-pure in $\OG$.
Since $e(\chi)\in$ $\OH$, and hence $z\in$ $\OH$, 
it suffices to show that $\OH z$ is $\CO$-pure in $\OH$.
For any $a$ such that $0\leq a\leq$ $2^n-1$ we have
$$s^a z= \lambda (s^ae(\eta)+s^a e(\bar\eta))=
\lambda (\zeta^ae(\eta)+\zeta^{-a}e(\bar\eta))\ .$$
We claim that the set $\{ z, 2^n e(\eta)\}$ is an
$\CO$-basis of $\OH z$. Note first that by \ref{Aechisymm}
$$sz-\zeta^{-1}z=\zeta^{-1}(\zeta^2-1)\lambda e(\eta)=
\mu 2^n e(\eta)$$
for some $\mu\in$ $\CO^{\times}$. Thus both $z$ and $2^n e(\eta)$ belong to $\OH z$.
We need to show that any of the
elements $s^a z$ with $a$ as before is an $\CO$-linear
combination of $z$ and $2^n e(\eta)$. We have
$$s^az-\zeta^{-a}z=\lambda\zeta^{-a}(\zeta^{2a})e(\eta)=
\nu 2^n e(\eta)$$
for some $\nu\in$ $\CO$, whence the claim.
It remains to show that $\OH z$ is $\CO$-pure in $\OH$.
Let $u=$ $\sum_{a=0}^{2^n-1}  \mu_a s^a$ be an element in $\OH z$
such that all coefficients $\mu_a\in$ $\CO$ are divisible by $1-\zeta$.
Write $u=$ $\beta z + \gamma 2^n e(\eta)$ with $\beta$, $\gamma\in$ $\CO$.
We need to show that $\beta$, $\gamma$ are divisible by $1-\zeta$.
Comparing coefficients for the two expressions of $u$ above yields
$$\mu_a = \beta \frac{\zeta^a+\zeta^{-a}}{(1-\zeta)^2} + \gamma \zeta^a$$
for $0\leq a\leq 2^n-1$. If $a=$ $2^{n-2}$, then $\zeta^a+\zeta^{-a}=$ $0$,
hence $\mu_a=$ $\gamma\zeta^a$, which implies that $\gamma$ is divisible
by $1-\zeta$, as $\mu_a$ is divisible by $1-\zeta$. 
We consider the above equation for $a=$ $1$, which reads
$$\mu_1 = \beta \frac{\zeta+\zeta^{-1}}{(1-\zeta)^2} + \gamma \zeta\ .$$
By \ref{zetaazetab} we have
$\frac{\zeta+\zeta^{-1}}{(1-\zeta)^2}\in$ $\CO^\times$.
Since $\mu_1$ and $\gamma$ are divisible by $1-\zeta$, this implies that
$\beta$ is divisible by $1-\zeta$. Thus $\OH z$ is $\CO$-pure in $\OH$,
and hence $\OG e(\chi)$ is symmetric. 

\medskip
Let finally $G$ be quasidihedral; that is, conjugation by the involution $t$ 
sends $s$ to $s^{-1+{2^{n-1}}}$. The calculations are similar to the 
previous case; we sketch the modifications.  Let $\bar\eta$ be the 
character of $H$  sending $s$ to $\zeta^{-1+2^{n-1}}=$ $-\zeta^{-1}$. 
Then $e(\chi)=$ $e(\eta)+e(\bar\eta)=$ 
$\sum_{a=0}^{2^n-1}\ (\zeta^a+(-1)^a\zeta^{-a})$.
As before, setting $z=$ $\frac{2^{n-1}}{(1-\zeta)^2} e(\chi)$, it suffices 
to  show that $\OH z$ is $\CO$-pure in $\OH$. One 
verifies as before,  that $\{ z,2^ne(\eta)\}$ is an $\CO$-basis of $\OH z$.
Let $u=$ $\sum_{a=0}^{2^n-1}  \mu_a s^a$ be an element in $\OH z$
such that all coefficients $\mu_a\in$ $\CO$ are divisible by $1-\zeta$.
Write $u=$ $\beta z + \gamma 2^n e(\eta)$ with $\beta$, $\gamma\in$ 
$\CO$. We need to show that $\beta$, $\gamma$ are divisible by 
$1-\zeta$. Comparing coefficients yields $\mu_a=$ 
$\beta\frac{\zeta^a+(-1)^a\zeta^{-a}}{(1-\zeta)^2}+$ $\gamma\zeta^a$.
If $a=$ $2^{n-2}$, then $a$ is even and
$\zeta^a+\zeta^{-a}=$ $0$, implying $\mu_a=$ $\gamma\zeta^a$, which
in turn implies that $\gamma$ is divisible by $1-\zeta$. Comparing coefficients
for $a=$ $1$ yields that $\beta$ is divisible by $1-\zeta$. Thus $\OH z$ is 
$\CO$-pure in $\OH$, and hence $\OG e(\chi)$ is symmetric. 
 This completes the proof of \ref{cyclicindexp}.

\section{On symmetric subalgebras of matrix algebras}

Let $G$ be a finite group. By the  remarks following \ref{OpureLemma}, 
the $\CO$-free $\CO$-algebra quotients of $\OG$ correspond to $\CO$-pure 
ideals, hence to subsets of $\Irr_K(G)$, and  by  \ref{AEndRemark} any 
$\CO$-free quotient of $\OG$ is the image of a structural map 
$\alpha : \OG \to $ $ \End_\CO(V)$ for some finitely generated $\CO$-free 
$\OG$-module $V$.  Moreover, the kernel of this map depends only on the 
set of irreducible characters of $G$ arising as constituents of the 
character of $V$. If the character $\chi$ of $V$ is irreducible, then 
the image of $\alpha$ is isomorphic to the algebra $\OG e(\chi)$, where
$e(\chi)=$ $\frac{\chi(1)}{|G|}\sum_{x\in G} \chi(x^{-1}) x$.
In that case, $\Im(\alpha)\cong$ $\OG e(\chi)$ has the same $\CO$-rank 
$\chi(1)^2$ as $\End_\CO(V)$. If $\chi\in$ $\Irr_K(G)$ has degree one, 
then $\OG e(\chi)\cong$ $\CO$ is trivially symmetric. If $\chi$ has 
defect zero, then $\OG e(\chi)$ is a matrix algebra over $\CO$, hence 
also symmetric. Both examples are special cases of the following 
well-known situation:

\begin{Proposition} \label{OGchimatrix}
Let $G$ be a finite group, $K$ a splitting field for $G$, and 
$\chi\in$ $\Irr_K(G)$. The algebra $\OG e(\chi)$ is isomorphic
to a matrix algebra over $\CO$ if and only if $\chi$ lifts
an irreducible Brauer character.
\end{Proposition}

\begin{proof}
We include a proof for the convenience of the reader.
Let $X$ be an $\CO$-free $\OG$-module with character $\chi$.
The character $\chi$ lifts an irreducible Brauer character
if and only if $k\tenO X$ is a simple $kG$-module.
The $kG$-module $k\tenO X$ is simple if and only if
the structural map $kG\to$ $\End_k(k\tenO X)$ is surjective.
By Nakayama's Lemma, this is the case if and only if the
structural map $\OG\to$ $\End_\CO(X)$ is surjective.
The result follows.
\end{proof}

If $\OG e(\chi)$ is symmetric but not a matrix algebra, then the 
following observation narrows down the possible symmetrising forms.

\begin{Proposition} \label{EndOVsubalgebras}
Let $V$ be a free $\CO$-module of finite rank $n$, and let $A$ be
a symmetric subalgebra of $\End_\CO(V)$ of rank $n^2$. Then
$Z(A)\cong$ $\CO$, and there is an integer $r\geq$ $0$ such that 
the restriction to $A$ of the map $\pi^{-r}\tr_V : \End_\CO(V)\to$ 
$K$ sends $A$ to $\CO$ and induces a symmetrising form on $A$.
Moreover, $r$ is the smallest nonnegative integer satisfying
$\pi^r\End_\CO(V)\subseteq$ $A$.
\end{Proposition}

\begin{proof}
Since the $\CO$-rank of $A$ is $n^2$, we have $K\tenO A\cong$
$\End_K(K\tenO V)$, whence $Z(A)\cong$ $\CO$. Any symmetrising
form on $\End_K(K\tenO V)$ is a nonzero linear multiple of
the trace map $\tr_{K\tenO V}$, and hence a symmetrising form on
$A$ is of the form $\pi^{-r}\tr_V$ for some integer $r$ such that
$\pi^{-r}\tr_V(A)=$ $\CO$. Since $\tr_V(A)\subseteq$ $\CO$, this forces 
$r\geq$ $0$. Let $s$ be the smallest nonnegative integer satisfying 
$\pi^s\End_\CO(V)\subseteq$ $A$. 
Let $e$ be a primitive idempotent in $\End_\CO(V)$.
Then $\tr_V(e)=$ $1$. We have $\pi^se\in$ $A$, hence
$\pi^{-r}\tr_V(\pi^se)=$ $\pi^{s-r}\in$ $\CO$. This
implies that $s\geq$ $r$. Since $\pi^{s-1}\End_\CO(V)$ is
not contained in $A$, there is an element $c\in$ $A$
such that $\pi^{-1}c\not\in$ $A$ and $c\in $ $\pi^s\End_\CO(V)$.
Thus $\CO c$ is a pure $\CO$-submodule of $A$. Thus there is an
$\CO$-basis of $A$ containing $c$. By considering the dual $\CO$-basis
with respect to the symmetrising form $\pi^{-r}\tr_V$, it follows that
$\pi^{-r}\tr_V(cA)=$ $\CO$. Since $c\in$ $\pi^s\End_\CO(V)$,
we have $\tr_V(cA)\subseteq$ $\pi^s\CO$, hence
$\pi^{-r}\tr_V(cA)\subseteq$ $\pi^{s-r}\CO$. 
This yields the inequality $s\leq$ $r$, whence the equality
$s=$ $r$.
\end{proof}

\begin{Corollary} \label{EndOVidempotents}
Let $V$ be a free $\CO$-module of finite rank $n$, and let $A$ be
a proper symmetric subalgebra of $\End_\CO(V)$ of rank $n^2$. Then
for any idempotent $i\in$ $A$, the integer $\tr_V(i)$ is divisible
by $p$. In particular, $p$ divides $n$.
\end{Corollary}

\begin{proof}
If $A$ is a proper symmetric subalgebra of $\End_\CO(V)$, then
the smallest nonnegative integer $r$ satisfying 
$\pi^r\End_\CO(V)\subseteq$ $A$ is positive. It follows from 
\ref{EndOVsubalgebras} that $A$ has a symmetrising form which is the 
restriction to $A$ of $\pi^{-r}\tr_V$ for some positive integer $r$.
In particular, if $i$ is an idempotent in $A$, then $\pi^{-r}\tr_V(i)$
is an element in $\CO$. Since $\tr_V(i)$ is an integer and $r>0$, this
implies that $p$ divides $\tr_V(i)$, Applied to $i=$ $1_A=$ $\Id_V$
yields that $p$ divides $n=$ $\tr_V(\Id_V)$.
\end{proof}

\begin{Corollary} \label{EndOVpprime}
Let $V$ be a free $\CO$-module of finite rank $n$. Suppose that
$(n,p)=$ $1$. Then $\End_\CO(V)$ has no proper symmetric subalgebra
of $\CO$-rank $n^2$.
\end{Corollary}

\begin{proof} This is clear by \ref{EndOVidempotents}.
\end{proof}

\begin{Corollary} \label{EndOVp}
Let $V$ be a free $\CO$-module of finite rank $p$. 
Every proper symmetric subalgebra of rank $p^2$ of
$\End_\CO(V)$ is local.
\end{Corollary}

\begin{proof}
If $i$ is an idempotent in a proper symmetric subalgebra $A$
of $\End_\CO(V)$, then $p$ divides $\tr_V(i)\leq$ $\tr_V(\Id_V)=$ $p$.
Thus $i=$ $\Id_V$. The result follows.
\end{proof}

In order to prove a slightly more general version of 
Theorem \ref{decnumberp}, we use the following notation. Let
$A$ be an $\CO$-algebra which is finitely generated free as
an $\CO$-module, such that $K\tenO A$ is split semisimple and
$k\tenO A$ is split. Let $X$ be a simple $K\tenO A$-module.
Then there is an $\CO$-free $A$-module $V$ such that $K\tenO V\cong$
$X$. It is well-known that for any simple $k\tenO A$-module $S$
the number $d^X_S$ of composition factors isomorphic to $S$ in
a composition series of $k\tenO V$ does not depend on the choice
of $V$. This follows, for instance, from the fact that 
$d^X_S=$ $\dim_K(iX)$, where $i$ is a primitive idempotent in $A$
such that $Ai/J(A)i\cong$ $S$. We denote by $e(X)$ the central primitive
idempotent in $K\tenO A$ which acts as identity on $X$.
We say that $X$ lifts a simple
$k\tenO A$-module $S$ if $k\tenO V\cong$ $S$, or equivalently, if
$d^X_S=$ $1$ and $d^X_{S'}=$ $0$ for any simple $k\tenO A$-module
$S'$ not isomorphic to $S$. Theorem \ref{decnumberp} is the
special case $A=$ $\OG$ of the following result.

\begin{Theorem} \label{Adecnumberp}
Let $A$ be an $\CO$-algebra which is finitely generated free as
an $\CO$-module, such that $K\tenO A$ is split semisimple and
$k\tenO A$ is split. Let $X$ be a simple $K\tenO A$-module.
Suppose that the $\CO$-algebra $Ae(X)$ is symmetric. Then
either $X$ lifts a simple $k\tenO A$-module, or $d^X_S$ is
divisible by $p$ for any simple $k\tenO A$-module $S$.
\end{Theorem}

\begin{proof}
Let $V$ be an $\CO$-free $A$-module such that $K\tenO V\cong$ $X$.
Suppose that $X$ does not lift a simple $k\tenO A$-module; that is,
$k\tenO V$ is not simple. Thus the image, denoted by $E$, of the 
structural map $A\to$ $\End_\CO(V)$ is a proper subalgebra of 
$\End_\CO(V)$. Note that $E\cong$ $A e(X)$. Moreover, for any 
idempotent $i\in$ $A$ we have $\tr_V(i)=$ $\rank_\CO(iV)=$ $\dim_K(iX)$.
As $A e(X)$ is assumed to be symmetric, it follows from 
\ref{EndOVidempotents} that for any idempotent $i\in$ $A$, the 
integer $\dim_K(iX)$ is divisible by $p$. Applied to primitive
idempotents, this shows that $d^X_S$ is divisible by $p$ for
any simple $k\tenO A$-module $S$.
\end{proof}

\section{Proofs of Propositions \ref{IOGpure} and \ref{Bechipure}}
\label{proofSection}

If $\chi : G\to$ $\CO^\times$ is a linear character, then
the corresponding $\CO$-pure ideal $I_\chi=$ $KG e(\chi)\cap\OG$ has
$\CO$-rank $1$ and is equal to $\CO(\sum_{x\in G}\chi({x^{-1}})x)$.
There is a unique $\CO$-algebra automorphism of $\OG$ sending $x\in$ $G$
to $\chi(x)x$. This automorphism sends $\sum_{x\in G}\chi(x^{-1})x$
to $\sum_{x\in G} x$. Thus the linear characters of $\OG$ are permuted 
transitively by the group of $\CO$-algebra automorphisms of $\OG$,
and therefore, in order to address the question whether $\OG/I_\chi$ is
symmetric, it suffices to consider the case where $\chi=$ $1$ is
the trivial character, in which case the corresponding pure ideal
is $I_1=$ $\OG(\sum_{x\in G} x)=$  $\CO (\sum_{x\in G}x)$. 
The annihilator of $I_1$ in $\OG$ is the augmentation ideal
$I(\OG)$. Combining these observations with Proposition
\ref{symquotients} and some block theory yields 
a proof of \ref{IOGpure}, which we restate in a slightly
more precise way.

\begin{Proposition} \label{IOGpurebis}
Let $G$ be a finite group. The following are equivalent.

\smallskip\noindent (i)
The $\CO$-algebra $\OG/\CO(\sum_{x\in G}x)$ is symmetric.

\smallskip\noindent (ii)
There exists an element $z\in$ $Z(\OG)$ such that $I(\OG)=$ $\OG z$.

\smallskip\noindent (iii)
The group $G$ is $p$-nilpotent and has a cyclic Sylow-$p$-subgroup.
\end{Proposition}

\begin{proof}
The equivalence of (i) and (ii) is clear by \ref{symquotients}.
Suppose that (ii) holds. Let $b$ be the principal block idempotent
of $\OG$. Then $I(\OG)b=$ $\OG zb$ is a proper ideal in $\OGb$; in
particular, $zb$ is not invertible in $Z(\OGb)$. Since $Z(\OGb)$ is 
local, it follows that $zb$ is in the radical of $Z(\OGb)$, and
hence the ideal $\OG zb$ is  contained in $J(\OGb)$. Since  
$\OGb/I(\OG)b\cong$ $\OG/I(\OG)\cong$ $\CO$, this implies that $\OGb$  
is a local algebra, and that $kG\bar z\bar b=$ $J(kG\bar b)$, where 
$\bar z$, $\bar b$ are the canonical images of $z$, $b$ in $kG$,
respectively. It follows  that the finite group $G$ is $p$-nilpotent
(see  \cite[Chapter 5, Theorems 8.1 and 8.2]{NagTsu}). Since 
$J(kG \bar b)=$ $kG\bar z\bar b$ is a principal ideal, it follows from 
a result of Nakayama \cite{Nak40} that $kG\bar b$ is uniserial, and hence 
$P$ is cyclic. Thus (ii) implies (iii).
Conversely, if (iii) holds, then the principal block algebra $\OGb$
of $\OG$ is isomorphic to $\OP$, where $P$ is a Sylow-$p$-subgroup
of $G$, and if $y$ is a generator of $P$, then $I(\OP)=$ $\OP(y-1)$.
Note that $I(\OG)$ contains all nonprincipal block algebras of $\OG$.
Set $z=$ $(y-1)b+\sum_{b'} \ b'$, where in the sum $b'$ runs
over all nonprincipal block idempotents. By the above, this is an 
element in $Z(\OG)$ satisfying $I(\OG)=$ $\OG z$, completing the proof.
\end{proof}

\begin{proof}[{Proof of Corollary \ref{IZGpure}}]
Suppose that $\Z G/\Z(\sum_{x\in G}x)$ is symmetric. Tensoring by $\CO$
over $\Z$ implies that the $\CO$-algebra $\OG/\CO(\sum_{x\in G}x)$ is 
symmetric. Thus $G$ is $p$-nilpotent with a cyclic $p$-Sylow subgroup.
This holds for any prime $p$. It follows that $G$ is a direct product of
its Sylow subgroups all of which are cyclic, and hence $G$ is cyclic.
Suppose conversely that $G$ is cyclic. Let $y$ be a generator of $G$. 
Then the augmentation ideal of $\Z G$ is equal to $(y-1)\Z G$, and this 
is also equal to the annihilator of $\Z(\sum_{x\in G}x)$. It follows 
from \ref{symquotients}, applied with the Noetherian ring $\Z$
instead of $\CO$, that $\Z G/\Z(\sum_{x\in G}x)$ is symmetric. 
\end{proof}

\begin{proof}[Proof of Proposition \ref{Bechipure}]
Suppose that $B/I$ is symmetric. By \ref{symquotients} the annihilator
$J$ of $I$ is of the form $J=$ $Bz$ for some $z\in$ $Z(B)$.
Let $X$ be an $\CO$-free $B$-module with character $\chi$. By the
assumptions, $\bar X=$ $k\tenO X$ is a simple module over
$\bar B=$ $k\tenO B$.
By \ref{OGchimatrix}, and using that $k$ is large enough, the structural 
map $B\to$ $\End_\CO(X)$ is surjective, and the kernel of this map is $J$. 
Since $J=$ $Bz$ is a 
proper ideal in $B$ and $Z(B)$ is local, it follows that $z\in$ $J(B)$. 
Thus the image $\bar J$ of $J$ in $\bar B$ is equal to $J(\bar B)$, 
and $\bar B$ has a single isomorphism class of simple modules.
Since $J(\bar B)=$ $\bar B\bar z$, where $\bar z$ is the image of
$z$ in $Z(\bar B)$, it follows as before from \cite{Nak40} that $\bar B$ 
is uniserial. 
Thus $B$ has cyclic defect groups. A block with cyclic defect and
a single isomorphism class of simple modules is nilpotent, which
shows that (i) implies (ii). Conversely, if (ii) holds, then $B$ 
is Morita equivalent to $\OP$, where $P$ is a defect group of $B$
(and $P$ is cyclic by the assumptions). 
Since $\chi$ lifts an irreducible Brauer character it follows that 
under some Morita equivalence, $\chi$ corresponds to the trivial 
character of $\OP$, and hence (i) holds by \ref{IOGpure} applied to $P$. 
\end{proof}

\section{An Example} \label{ExamplesSection}

As was pointed out after \ref{centraltype}, all irreducible characters 
of finite $p$-groups of order at most $p^3$ have the symmetric quotient 
property. The next example shows that for any odd prime $p$ there is
a finite $p$-group of order $p^{p+1}$ having at least one irreducible 
character which does not have the symmetric quotient property.

\begin{Example} \label{wreathExample}
Let $p$ be an odd prime. $G=$ $Q\wr R=$ $H\rtimes R$, where $Q$, $R$ are
cyclic of order $p$, and where $H$ is a direct product of $p$ copies
of $Q$ which are transitively permuted by $R$. 
Let $s$ be a generator of $Q$ and $\zeta$ be a primitive $p$-th 
root of unity. For $1 \leq$ $i\leq$ $p$ let $\psi_i : H\to$ $\CO^\times$
be the linear character sending $(s^{a_1},s^{a_2},..,s^{a_p})\in$ $H$
to $\zeta^{a_i}$; that is, the kernel of $\psi_i$ contains all but
the $i$-th copy of $Q$ in $H$, the $\psi_i$ are pairwise different,
and they are permuted transitively by the action of $G$. Set $\chi=$ 
$\Ind^G_H(\psi_1)$. Then $\chi\in$ $\Irr_K(G)$, and the $\CO$-algebra 
$\OG e(\chi)$ is not symmetric. 

\smallskip
The irreducibility of $\chi$ is a standard result.
In order to show the symmetry of $\OG e(\chi)$, observe first that 
$\Res^G_H(\chi)=$ $\sum_{i=1}^p \psi_i$ because the $\psi_i$ form a 
$G$-orbit in $\Irr_K(H)$. We have
$$e(\chi)= \sum_{i=1}^p\ e(\psi_i) = 
\frac{1}{|H|}\sum_{h\in H} (\sum_{i=1}^p \psi_i(h^{-1})) h\ .$$
The coefficients $\sum_{i=1}^p \psi_i(h^{-1})$ are divisible by
$1-\zeta$ because they are sums of $p$ (arbitrary) powers of 
$\zeta$. Moreover, for $h=$ $(s,1,..,1)\in$ $H$ one sees that $1-\zeta$
is the highest power of $1-\zeta$ dividing this coefficient.
Thus if $\OG e(\chi)$ were symmetric, then $\OG z$ would have to
pure in $\OG$, where $z =$ $\frac{|H|}{1-\zeta} e(\chi)$.
Since $z\in$ $\OH$, this is the case if and only if $\OH z$ is
pure in $\OH$. We will show that $\OH z$ is not pure in $\OH$.
If $u=$ $(s^{a_1},s^{a_2},..,s^{a_p})\in$ $H$, then
$$uz = \frac{|H|}{1-\zeta} \sum_{i=1}^p u e(\psi_i) =
\frac{|H|}{1-\zeta} \sum_{i=1}^p \zeta^{a_i} e(\psi_i)\ .$$
Applied to the identity element and
$v=$ $(1,1,..1,s,1,..1)$, with $s$ in the $i$-th component, 
and taking the difference yields $z-vz=$
$|H|e(\psi_i)\in$ $\OG z$. For any $u\in$ $H$, the above formula
yields $uz-z\in$ $\oplus_{i=1}^p\ \CO |H|e(\psi)$. Thus the set
$$\{ z, |H|e(\psi_i)\ (2\leq i\leq p)\}$$
is an $\CO$-basis of $\OH z$. Since $p$ is odd, this basis has
at least three elements. Suppose that $w=$ $\sum_{h\in H} \mu_h h$ 
is an element in $\OH z$. Write 
$w=$ $\alpha z+\sum_{i=2}^p \beta_i |H|e(\psi_i)$
with $\alpha$, $\beta_i\in$ $\CO$. Thus
$$\mu_h = \alpha \frac{\sum_{i=1}^p\ \psi(h^{-1})}{1-\zeta}\ +\
  \sum_{i=2}^p\ \beta_i\psi_i(h^{-1})\ .$$
Note that if $\sum_{i=2}^p\beta_i$ is divisible by $1-\zeta$ then
any sum of the form $\sum_{i=2}^p\ \beta_i\psi_i(h^{-1})$ is divisible
by $1-\zeta$ because any character value $\psi_i(h^{-1})$ is
a power of $\zeta$. This shows that if $1-\zeta$ divides
both $\alpha$ and the sum $\sum_{i=2}^p\beta_i$, then $1-\zeta$ divides
$\mu_h$ for all $h\in$ $H$. But since $p>2$ we may choose invertible
elements $\beta_i$ satisfying $\sum_{i=2}^p\beta_i=$ $0$. This shows
that even if all $\mu_h$ are divisible by $1-\zeta$, this does not
imply that $\alpha$ and all $\beta_i$ are divisible by $1-\zeta$,
hence $\OH z$ is not $\CO$-pure in $\OH$.
\end{Example}

\end{document}